\documentclass[11pt]{amsart}

\usepackage{amsmath,amsthm, amscd, amssymb, amsfonts}
\newcommand{\ca}{{\mathcal{C} }}

\newcommand{\D}{{\mathcal{D} }}

\newcommand{\C}{{\mathcal{C} }}

\newcommand{\id}{\mbox{\rm id\,}}

\newcommand{\Aut}{\mbox{\rm Aut\,}}

\newcommand\Hom{\operatorname{Hom}}

\theoremstyle{plain}

\numberwithin{equation}{section}

\newtheorem{theorem}{Theorem}[section]

\newtheorem{proposition}[theorem]{Proposition}

\theoremstyle{definition}

\theoremstyle{remark}

\newtheorem{remark}[theorem]{Remark}

\def\pf{\begin{proof}}

\def\epf{\end{proof}}

\theoremstyle{remark}

\begin{document}

\title{On braided and ribbon unitary fusion categories}
\author[C\'esar Galindo]{C\'esar Galindo}
\thanks{ This work  was partially supported by Vicerrector\'ia de Investigaciones de la Universidad de los Andes}
\address{ Departamento de Matem\'aticas, Universidad de los Andes,
\newline \indent Bogot\'a, Colombia}
\email{cn.galindo1116@uniandes.edu.co, cesarneyit@gmail.com}
\begin{abstract}
We prove that every braiding on a unitary fusion category is automatically unitary, and that every unitary braided fusion category admits a unique unitary ribbon structure.
\end{abstract}


\date{\today}
\maketitle

\section{Introduction}

A unitary braided fusion category (UBFC) is a braided fusion category (BFC) over the complex numbers, where the $\Hom$-spaces have a Hilbert structure compatible with the tensor product and the braiding (see Subsection \ref{subsection unitary fusion}). Unitarity is the key additional structure for some applications in mathematics and theoretical physics. For example, in mathematics, UBFCs give rise to unitary representations of the Artin braid group and finite-depth subfactors of the hyperfinite II$_1$ von Neumann factor \cite{Wenzl}. For physics, a unitary structure is important in order to construct \emph{unitary} (2+1)-dimensional TQFTs \cite{TurBook}\cite{Witten} and algebraic models for exotic 2-dimensional physical (anyonic) systems \cite{QHall}\cite{MR}. These last two applications make unitary modular categories the mathematical foundation for topological quantum computation \cite{Freedman}\cite{Kitael}\cite{WangBook}. The above applications have renewed interest in the construction and study of UBFC properties.

Fusion categories are ``quantum analogues'' of finite groups, mainly because the prototypical example of a fusion category is Rep${\mathbb C}(G)$, the category of finite-dimensional complex representations of a finite group $G$. This analogy also holds because in fusion categories, there are phenomena such as nilpotency \cite{GeNik}, solvability \cite{ENO2}, and simplicity \cite{ENO3}. However, important differences exist between general fusion categories and Rep${\mathbb C}(G)$, one of which is unitarity. The fusion category Rep$_{\mathbb C}(G)$ admits a canonical unitary structure, but, for example, the Yang-Lee category does not \cite{ENO}. In fact, there are families of premodular categories that do not admit unitary structures at all \cite{Rowell}. On the other hand, there are at least two infinite families of unitary premodular categories. One is associated with quantum groups \cite{Wenzl}, and the other is related to finite groups \cite{ENO2}. The fusion categories of the latter family always admit a unique unitary structure \cite{GHR}.

In this note, we are interested in the following natural questions: Does a BFC admit a unitary structure if the underlying fusion category is unitary? If it does, how many unitary ribbon structures does it admit? We prove that every braiding over a unitary fusion category is automatically unitary (see Theorem \ref{braids are unitary})\footnote{This answers Problem 3.3 in http://aimpl.org/fusioncat/, posted by Zhenghan Wang.}, and every UBFC admits a unique unitary ribbon structure (see Theorem \ref{unique unitary ribbon}).

\medbreak
\textbf{Acknowledgements:} The author would like to thank Paul Bressler, Eric Rowell, and Zhenghan Wang for their useful comments. The author is also grateful to Tian Lan for identifying an error in Proposition 3.1 of the previous version.

\section{Preliminaries}

In this note, we will use the basic theory of fusion categories and braided fusion categories. For further details on these topics, we refer the reader to \cite{ENO}. In this section, we recall some definitions and results on unitary fusion categories. Much of the material presented here can be found in \cite{TurBook}.

\subsection{Unitary fusion categories}\label{subsection unitary fusion}

A \textbf{$C^*$-category} $\D$ is a $\mathbb{C}$-linear abelian category with an involutive antilinear contravariant endofunctor $*$ that is the identity on objects. The hom-spaces $\Hom_\D(X, Y)$ are Hilbert spaces, and the norms satisfy 
\[
||fg|| \leq ||f||\ ||g||, \quad ||f^*f|| = ||f||^2,
\]
for all $f \in \Hom_\D(X, Y)$, $g \in \Hom_\D(Y, Z)$, where $f^*$ denotes the image of $f$ under $*$.

Let $X$ and $Y$ be objects in a $C^*$-category. A morphism $u: X \to Y$ is \textbf{unitary} if $uu^* = \id_Y$ and $u^*u = \id_X$. A morphism $a: X \to X$ is \textbf{self-adjoint} if $a^* = a$.

\begin{remark}\label{remark polar}
Every isomorphism in a $C^*$-category has a polar decomposition, \emph{i.e.}, if $f: X \to Y$ is an isomorphism, then $f = ua$, where $a: X \to X$ is self-adjoint and $u: X \to Y$ is unitary. See \cite[Proposition 8]{Baez}.
\end{remark}

A \textbf{unitary fusion category} is a fusion category $\C$, where $\C$ is a $C^*$-category, the constraints are unitary, and $(f \otimes g)^* = f^* \otimes g^*$ for every pair of morphisms $f, g$ in $\C$.

\begin{remark}{~}
\begin{enumerate}
    \item A unitary fusion category is a fusion category with additional structure. Hence, a fusion category could have more than one unitary structure. All examples known to the author admit a unique unitary structure. Moreover, in \cite[Theorem 5.20]{GHR}, it was proved that every weakly group-theoretical fusion category admits a unique unitary structure.
    
    \item If $\C$ is a unitary fusion category, we can find bases such that the $F$-matrices $(F^{ijk}_l)_{n,m} = F^{i,j,k}_{l;n,m}$ are unitary, where $\{F^{i,j,k}_{l;n,m}\}$ are the $6j$-symbols (see \cite{WangBook} or \cite{TurBook} for the definition of $6j$-symbols). Conversely, if for a fusion category $\C$ it is possible to find bases such that the $F$-matrices $(F^{ijk}_l)_{n,m} = F^{i,j,k}_{l;n,m}$ are unitary, then $\C$ is a unitary fusion category. See \cite[Section 4]{Yama}.
\end{enumerate}
\end{remark}

\section{Braiding and modular structures over unitary fusion categories are unitary}

\subsection{The center of a unitary fusion category}

We shall recall the definition of the \textbf{center} $\mathcal{Z}(\ca)$ of a monoidal category $\ca$, see \cite[Chapter XIII]{Kas}. The objects of $\mathcal{Z}(\ca)$ are pairs $(Y, c_{-,Y})$, where $Y \in \ca$ and $c_{X,Y} : X \otimes Y \to Y \otimes X$ are isomorphisms natural in $X$ satisfying 
\[
c_{X \otimes Y, Z} = (c_{X,Z} \otimes \id_Y)(\id_X \otimes c_{Y,Z}), \quad c_{I,Y} = \id_Y,
\]
for all $X, Y, Z \in \ca$. A morphism $f: (X, c_{-,X}) \to (Y, c_{-,Y})$ is a morphism $f: X \to Y$ in $\ca$ such that 
\[
(f \otimes \id_W)c_{W,X} = c_{W,Y}(\id_W \otimes f),
\]
for all $W \in \ca$.

The center is a braided monoidal category with structure given as follows:
\begin{itemize}
  \item The tensor product is $(Y, c_{-,Y}) \otimes (Z, c_{-,Z}) = (Y \otimes Z, c_{-, Y \otimes Z})$, where 
  \[
  c_{X, Y \otimes Z} = (\id_Y \otimes c_{X,Z})(c_{X,Y} \otimes \id_Z): X \otimes Y \otimes Z \to Y \otimes Z \otimes X,
  \]
  for all $X \in \ca$.
  \item The identity element is $(I, c_{-,I})$, where $c_{Z,I} = \id_Z$.
  \item The braiding is given by the morphism $c_{X,Y}$.
\end{itemize}

If $\C$ is a unitary fusion category, the \textbf{unitary center} $\mathcal{Z}^*(\C)$ is defined as the full tensor subcategory of $\mathcal{Z}(\C)$, where $(X, c_{-,X}) \in \mathcal{Z}^*(\C)$ if and only if $c_{W,X}: W \otimes X \to X \otimes W$ is unitary for all $W \in \C$.

\begin{remark}

In the published version of this work, it was stated that the inclusion functor $\mathcal{Z}^*(\C) \subseteq \mathcal{Z}(\C)$, being an equivalence as established in \cite[Proposition 5.24]{GHR} and \cite[Theorem 6.4]{Mueg2}, implied that $\mathcal{Z}^*(\C) = \mathcal{Z}(\C)$. However, thanks to an observation by Tian Lan, we now understand that, in general, $\mathcal{Z}^*(\C) \neq \mathcal{Z}(\C)$; in other words, $\mathcal{Z}^*(\C)$ is not a replete subcategory of $\mathcal{Z}(\C)$.

To illustrate this distinction, consider the following case, as pointed out by Tian Lan: $\operatorname{URep}(G)$, the category of unitary representations of a finite group $G$, and $\operatorname{Rep}_{\mathbb{C}}(G)$, the category of all complex representations of $G$. Although these two categories are equivalent as fusion categories, they are not identical in terms of their objects. Indeed, Tian Lan provided a counterexample using the category $\operatorname{Hilb}_G$ of finite-dimensional Hilbert spaces graded over $G$, where we have $\operatorname{URep}(G) \subsetneq \operatorname{Rep}(G) \subset \mathcal{Z}(\operatorname{Hilb}_G)$ and $\operatorname{URep}(G) \subset \mathcal{Z}^*(\operatorname{Hilb}_G)$, yet the inclusion $\operatorname{URep}(G) \subsetneq \operatorname{Rep}(G)$ is not a replete subcategory.

This nuance arises because every complex representation is unitarizable (i.e., isomorphic to a unitary representation). 
\end{remark}

Let $\C$ be a unitary fusion category equipped with a braiding $c_{X,Y}: X \otimes Y \to Y \otimes X$. We say that $c$ is a \textbf{unitary braiding} if, for every pair of objects $X, Y \in \C$, the morphism $c_{X,Y}$ is unitary.

\begin{theorem}\label{braids are unitary}
Every braiding of a unitary fusion category is unitary
\end{theorem}

\begin{proof}
Let $\C$ be a unitary fusion category, and let $c$ be a braiding on $\C$. The braiding $c$ induces a braided tensor inclusion 
\[
\iota_c: \C \hookrightarrow \mathcal{Z}(\C), \quad X \mapsto (X, c_{-,X}).
\]
By \cite[Proposition 5.24]{GHR} and \cite[Theorem 6.4]{Mueg2}, the inclusion functor $\mathcal{Z}^*(\C) \subseteq \mathcal{Z}(\C)$ is an equivalence of categories. Therefore, for every  object $X \in \C$, there exists an object $(Y, \tau_{-,Y}) \in \mathcal{Z}^*(\C)$ and an isomorphism 
\[
f : (X, c_{-,X}) \to (Y, \tau_{-,Y})
\]
in $\mathcal{Z}(\C)$.

Using the naturality of $c_{-,X}$, $c_{-,Y}$, and $\tau_{-,Y}$, it follows that $\tau_{-,Y} = c_{-,Y}$. Since $(Y, \tau_{-,Y}) \in \mathcal{Z}^*(\C)$, the isomorphisms $c_{S,Y}$ are unitary for all $S \in \C$.

By Remark \ref{remark polar}, we can write $f$ as $f = u a$, where $u : X \to Y$ is a unitary isomorphism. Hence, for every $S \in \C$,
\[
c_{S,X} = (  u\otimes \id_S )^*  c_{S,Y} ( \id_S\otimes u).
\]
Since $c_{S,Y}$ is unitary and $u$ is also unitary, it follows that $c_{S,X}$ is unitary for all $S\in \C$. This shows that every braiding in a unitary fusion category is unitary.
\end{proof}

\begin{remark}{~}
\begin{enumerate}
    \item Theorem \ref{braids are unitary} implies that if the $F$-matrices $(F^{ijk}_l)_{n,m} = F^{i,j,k}_{l;n,m}$ are unitary, then the $R$-matrices of the braiding are always unitarily diagonalizable.

    \item A Kac algebra $(H, m, \Delta, *)$ is a semisimple Hopf algebra such that $(H, *)$ is a $C^*$-algebra and the maps $\Delta$ and $\varepsilon$ are $C^*$-algebra maps. Theorem \ref{braids are unitary} implies that every $R$-matrix in a Kac algebra is unitary in the sense that $R^* = R^{-1}$.
\end{enumerate}
\end{remark}

\subsection{Ribbon structures on unitary fusion categories}

If $\C$ is a fusion category, then for every $f\in \Hom_{\C}(X,Y)$
the \textbf{transpose} of $f$, is defined by $$^tf:=(\id_{X^*}\otimes
\text{ev}_Y)(\id_{X^*}\otimes f\otimes \id_{Y^*})(\text{coev}_{X}\otimes \id_{Y^*})\in
\Hom_{\C}(Y^*,X^*).$$

A \textbf{twist} on a braided fusion category $\C$ is a natural
automorphism of the identity functor $\theta \in \Aut(\text{Id}_\C)$, such
that
$$ \theta_{X\otimes Y} = (\theta_X \otimes \theta_Y )c_{Y,X}c_{X,Y}$$
for all $X, Y \in \C$. A twist is called a \textbf{ribbon structure}
if $^t\theta_X = \theta_{X^*}$. A fusion category with a ribbon
structure is called a \textbf{ribbon fusion category}. Each ribbon
structure $\theta$ defines a \textbf{quantum dimension function} by
$\dim_\theta(X)= \text{ev}_X c_{X,X^*}(\theta_X\otimes \id_{X^*})
\text{coev}_X$.

We shall denote by $\Aut_{\otimes}(\text{Id}_\C)_{(+,-)}$ the abelian
group of tensor automorphisms $\gamma$ of the identity  such that
$\gamma_X=\pm \id_X$ for every simple object $X\in \C$.
\begin{proposition}\label{torsor}
Let $\C$ be a  braided fusion category. If the set of ribbon
structures is not empty, it is a torsor under
$\Aut_{\otimes}(\text{Id}_\C)_{(+,-)}$.
\end{proposition}
\begin{proof}
Let $\theta$ and $\theta'$ be ribbon structures. It is easy to see that
$\gamma:=\theta^{-1} \theta':\text{Id}_\C\to \text{Id}_\C$ is a tensor
automorphism of the identity. For every simple object, we have
$\theta_X= \theta(X)\id_X$, $\theta_X'= \theta(X)'\id_X$, $\gamma_X=
\gamma(X)\id_X$ for some $\gamma(X), \theta(X), \theta(X)'\in
\mathbb C^*$ and $\theta(X)'=\gamma(X)\theta(X)$. Since $\theta'$ is
a ribbon structure, for every simple object $X\in \C$,
$\dim_{\theta'}(X)=\dim_{\theta'}(X^*)$. On the other hand,
$\dim_{\theta'}(X)=\gamma(X)\dim_\theta(X)$. Therefore
$\gamma(X)=\gamma(X^*)$ and, since $\gamma(X^*)=\gamma(X)^{-1}$ we
conclude that $\gamma$ has order two.

Conversely, if $\gamma$ is an automorphism of the identity such that
$\gamma_X= \pm\id_X$ for every simple object, then, for every ribbon
structure $\theta$, the natural isomorphism $\theta'=\theta\gamma$
is a new ribbon structure.
\end{proof}

If $\C$ is a unitary fusion category a ribbon structure on $\C$ is
called \textbf{unitary ribbon structure} if $\theta_X$ is unitary,
$(\text{coev}_X)^*=\text{ev}_X( c_{X,X^*})(\theta_X\otimes
\id_{X^*})$ and $(\text{ev}_X)^*=(\id_{X^*}\otimes
\theta_X^{-1}) (c_{X^*,X})^{-1} \text{coev}_X$ for all $X\in
\C$. A unitary fusion category with a unitary ribbon structure is
called a \textbf{unitary ribbon fusion category} or \textbf{unitary premodular category}. In a unitary
ribbon fusion category $$\dim_\theta(X)= \text{ev}_X
(c_{X,X^*}) (\theta_X\otimes \id_{X^*}) \text{coev}_X =
(\text{coev}_X )^* \text{coev}_X,$$ therefore, the quantum
dimension of every object is a positive number.

\begin{theorem}\label{unique unitary ribbon}
Every  braided fusion category with a unitary structure admits a
unique unitary ribbon structure.
\end{theorem}

\begin{proof}
By \cite[Proposition 2.4]{Mueg} every braided unitary fusion
category admits a canonical unitary ribbon structure. Let $\theta_c$
the canonical ribbon structure associated to $c$. By Proposition
\ref{torsor}, if $\theta'$ is another unitary ribbon structure, then
there is $\gamma \in \Aut_{\otimes}(\text{Id}_\C)_{(+,-)}$ such that
$\theta'= \theta_c\gamma$. If $\gamma$ is not the identity there is
a simple object $X\in \C$ such that $\gamma_X=-\id_X$, then
$\dim_{\theta'}(X)=-\dim_{\theta_c}(X)<0$, but the quantum dimension
of every object of any unitary ribbon structure is positive.
Therefore $\gamma$ is the identity and $\theta_c$ is unique.
\end{proof}

\begin{remark}
It follows from Theorem \ref{unique unitary ribbon} that if a
unitary braided fusion category is non-degenerate (see \cite{DGNO}
for a definition), then it admits a unique unitary modular
structure.
\end{remark}

\makeatletter

 \def\@biblabel#1{#1.}

\makeatother

\end{document}